\documentclass[a4paper,12pt]{amsart}
\usepackage{amssymb}
\usepackage{amsmath}
\usepackage{array}
\usepackage{color}

\usepackage[notref,notcite,final]{showkeys} % comment out at the end, or add the option "final"
\newcommand{\hide}[1]{}
\newcommand{\comm}[1]{\marginpar{\tiny #1}}
\newcommand{\ignore}[1]{}

\newcommand{\C}{\mathbb{C}}

\newcommand{\Z}{\mathbb{Z}}
\newcommand{\Q}{\mathbb{Q}}

\newcommand{\F}{\mathbb{F}}

\newcommand{\Legendre}[2]{\genfrac{(}{)}{}{}{#1}{#2}}
\newcommand{\tLegendre}[2]{\genfrac{(}{)}{}{1}{#1}{#2}}

\def\Li{{\rm Li}}

\def\aa{\alpha}
\def\bb{\beta}

\newtheorem{dummy}{Dummy}

\newtheorem{lemma}[dummy]{Lemma}
\newtheorem{theorem}[dummy]{Theorem}
\newtheorem*{theorem*}{Theorem}
\newtheorem{prop}[dummy]{Proposition}

\theoremstyle{definition}

\theoremstyle{remark}

\newtheorem*{rem*}{Remark}

\begin{document}
\bibliographystyle{plain}

\author{Sandro Mattarei}
\email{\tt smattarei@lincoln.ac.uk}
\address{School of Mathematics and Physics,
University of Lincoln, Brayford Pool,
Lincoln, LN6 7TS, United Kingdom}

\author{Roberto Tauraso}
\email{tauraso@mat.uniroma2.it}
\address{Dipartimento di Matematica,
Universit\`a di Roma ``Tor Vergata'',
via della Ricerca Scientifica,
00133 Roma, Italy}

\title{From generating series to polynomial congruences}

\begin{abstract}
Consider an ordinary generating function $\sum_{k=0}^{\infty}c_kx^k$,
of an integer sequence of some combinatorial relevance, and assume that it admits a closed form $C(x)$.
Various instances are known where the corresponding truncated sum $\sum_{k=0}^{q-1}c_kx^k$,
with $q$ a power of a prime $p$, also admits a closed form representation when viewed modulo $p$.
Such a representation for the truncated sum modulo $p$ frequently bears a resemblance with the shape of $C(x)$,
despite being typically proved through independent arguments.
One of the simplest examples is the congruence
$\sum_{k=0}^{q-1}\binom{2k}{k}x^k\equiv(1-4x)^{(q-1)/2}\pmod{p}$
being a finite match for the well-known generating function
$\sum_{k=0}^\infty\binom{2k}{k}x^k=
1/\sqrt{1-4x}$.

We develop a method which allows one to directly infer the closed-form representation of the truncated sum
from the closed form of the series for a significant class of series
involving central binomial coefficients.
In particular, we collect various known such series whose closed-form representation involves polylogarithms
$\Li_d(x)=\sum_{k=1}^{\infty}x^k/k^d$, and after supplementing them with some new ones
we obtain closed-forms modulo $p$ for the corresponding truncated sums,
in terms of finite polylogarithms
$\pounds_d(x)=\sum_{k=1}^{p-1}x^k/k^d$.
\end{abstract}


\maketitle

\thispagestyle{empty}

\section{Introduction}

There is a vast and expanding literature on evaluating sums of combinatorial numbers such as binomial coefficients modulo a prime $p$,
or a power of a prime.
Among the simplest examples are
\begin{equation}\label{eq:sums_mod_p}
\sum_{k=0}^{q-1}\binom{2k}{k}\equiv
\Legendre{q}{3}\pmod{p},
\quad\text{and}\quad
\sum_{k=0}^{q-1}C_k\equiv
\frac{3\Legendre{q}{3}-1}{2}\pmod{p},
\end{equation}
where $q$ is a power of a prime $p$,
\[
C_k=\frac{1}{k+1}\binom{2k}{k}=\binom{2k}{k}-\binom{2k}{k+1}
\]
denotes the $k$th {\em Catalan number},
and $\Legendre{a}{3}$ is a Legendre symbol.
More general forms of congruences~\eqref{eq:sums_mod_p}, and variations, were established in~\cite[Theorem~1.2 and Corollary~1.3]{PanSun},
and later extended in various directions by several authors,
see for example~\cite[Theorem~1.1]{Sun:BCL},
but in this Introduction we take the simple formulations in Equations~\eqref{eq:sums_mod_p}
to illustrate the general principle which informs this paper.

Many such congruences happen to be specializations of more general polynomial congruences, such as
\begin{equation}\label{eq:central_pol}
\sum_{k=0}^{q-1}\binom{2k}{k}x^k\equiv(1-4x)^{(q-1)/2}\pmod{p}
\end{equation}
in case of the former of Equations~\eqref{eq:sums_mod_p},
see Theorem~\ref{thm:central_pol} below,
which can then easily be recovered by specializing at $x=1$.
Such polynomial formulations, when they exist, are manifestly superior to their numerical counterparts for other reasons
beyond allowing specialization at any integer values for $x$, or indeed algebraic integers or even $p$-integral algebraic numbers,
the most compelling being the possibility of producing other congruences by differentiation or integration, two workhorses of the
{\em generating functions} arsenal.

Equation~\eqref{eq:central_pol} bears a strong resemblance to the well-known power series identity
\begin{equation}\label{eq:central_gf}
\sum_{k=0}^\infty\binom{2k}{k}x^k
=
\frac{1}{\sqrt{1-4x}},
\end{equation}
which exhibits the {\em generating function} for the
{\em central binomial coefficients} $\binom{2k}{k}=(-4)^{k}\binom{-1/2}{k}$.
We contend that the most insightful way to prove congruences such as Equation~\eqref{eq:central_pol}
is to derive them from corresponding power series identities, in this case Equation~\eqref{eq:central_gf}.
In this paper we show that this indeed possible in a variety of cases where such similarities occur.

This quite general procedure, which we may call {\em truncation and reduction modulo $p$,}
may be achieved in simple cases by little more than an appropriate application of the congruence $(1+x)^{q}\equiv 1+x^q\pmod{p}$.
Incidentally, this also justifies why the range $0\le k<q$ generally appears to be a natural summation range to consider for the truncated sums.
Equation~\eqref{eq:central_pol} and its analogue for Catalan numbers will be obtained in this way in Section~\ref{sec:original}.
In Section~\ref{sec:shifted} we will show that {\em shifted} versions of Equations~\eqref{eq:sums_mod_p} given in~\cite{PanSun}, such as
\begin{equation}\label{eq:2k_k+d}
\sum_{k=0}^{q-1}\binom{2k}{k+d}\equiv
\Legendre{q-d}{3}\pmod{p}
\end{equation}
in case of the former, for any $0\le d<q$, are also specializations of polynomial congruences,
which can be similarly obtained from corresponding power series identities by truncation and reduction modulo $p$.

We devote the rest of the paper to more sophisticated polynomial congruences which can be obtained by our general method.
We focus on a class of series and sums, involving binomial coefficients and generalized harmonic numbers,
whose closed forms require polylogarithms and their finite analogues.
We describe here only the simplest illustrative example of our results which comes from the known generating function
\[
\sum_{k=1}^{\infty}\binom{2k}{k}\frac{x^k}{k}=
-2\log\left(\frac{1+\sqrt{1-4x}}{2}\right)
=2\Li_1(\beta),
\]
where $\beta=(1-\sqrt{1-4x})/2$ and
$\Li_d(x)=\sum_{k=1}^{\infty}x^k/k^d$
denote polylogarithms.
An application of our method of truncation and reduction modulo $p$, where $p$ is an odd prime, produces the congruence
\[
\sum_{k=1}^{p-1}\binom{2k}{k}\frac{x^k}{k}\equiv
\pounds_1(\beta)+\pounds_1(1-\beta)\pmod{p},
\]
where
$\pounds_d(x)=\sum_{k=1}^{p-1}x^k/k^d$
denotes a {\em finite polylogarithm,} see Section~\ref{sec:polylog} for further details.
This congruence is Equation~\eqref{N1} in Theorem~\ref{thm:2kk}.
Note that the two polylogarithms at the right-hand side of this congruence should be viewed as power series in $x$, but because of cancellation their sum modulo $p$
turns out to be a polynomial in $x$.
One may also view $\beta$ as the principal indeterminate
with $x=\beta(1-\beta)$ defined in terms of it.
We will adopt this point of view, so all our congruences will be between polynomials rather than power series,
and our proofs will run much smoother.

A similar truncation procedure can be applied to certain series of the form
$\sum_{k=1}^{\infty}\binom{2k}{k}x^k/k^d$
with higher $d$, but those naturally fit in a wider class of series which admit closed forms involving polylogarithms.
Roughly speaking, the series we consider here are generating series of sequences of the general form
$\binom{2k}{k}a_k$
or
$C_ka_k$,
where $a_k$ might be $1/k^d$, or a {\em generalized harmonic number}
$H_k^{(d)}=\sum_{j=1}^{k}1/j^d$
or possibly a linear combination of products of them.
Such series can be conveniently sorted by their {\em level,} which is the name we give to
the highest power of $k$ occurring in the denominator of the expression $a_k$ once expanded.
Thus, the series considered so far in this introduction have level zero or one.
In Section~\ref{sec:series} we exhibit evaluations in closed form for several series of level up to three,
quoting several from the literature and producing some new ones.
The closed forms of series of level $d$ involve polylogarithms $\Li_d$ (and possibly of lower level).
As an illustrative example we mention
\begin{equation*}
\sum_{k=1}^{\infty}C_kH_k^{(2)}x^{k+1}=
2\beta\,\Li_2\left(\beta\right)
-(1-\beta)\Li_1(\beta)^2,
\end{equation*}
which is our Equation~\eqref{S3} and gives a closed form for a series of level two.

In Section~\ref{sec:polynomials} we apply truncation and reduction modulo $p$ to all series of level up to three
considered in Section~\ref{sec:series}.
In the example mentioned above the identity for that series of level two leads to the polynomial congruence
\begin{equation*}
\sum_{k=1}^{p-1}C_k H_k^{(2)}x^{k+1}
\equiv 2\beta\pounds_2(\beta)+2(1-\beta)\pounds_2(1-\beta)
\pmod{p},
\end{equation*}
for $p>3$,
which is our Equation~\eqref{C11}.
The right-hand side of this congruence, as well as that of the congruence
introduced earlier, and all other congruences we produce in Section~\ref{sec:polynomials},
exhibits a symmetry with respect to interchanging $\beta$ and $1-\beta$.
This symmetry, which is a necessity as the left-hand side is invariant with respect to this substitution,
has no counterpart for the generating functions of the corresponding power series,
where this substitution would not even make sense.

The algebraic manipulations of Section~\ref{sec:polynomials} will require various functional equations (in the form of congruences) for finite polylogarithms,
which we collect in Section~\ref{sec:polylog} after recalling definitions and main properties.
This material is well known, and mostly traces back to Mirimanoff in some form~\cite[p.~61]{Mirimanoff}.
However, we provide two new proofs of a 4-term identity (congruence) for the finite logarithm $\pounds_1$
due to Kontsevich~\cite{Kontsevich}, which deduce it from the fundamental functional equation for the ordinary logarithmic function
by the same methods which inform the present paper.

Another preparatory section is Section~\ref{sec:transform}, where we discuss a certain involutory transform for sequences,
given in Equation~\eqref{eq:transform_general}, which we need in the sequel.
The transform itself is originally due to Euler, see~\cite[Equation~(1.20)]{Norlund},
but the known proof only works in an analytic context where the series involved are assumed to have a positive convergence radius,
as it depends on an integral formula for the Hadamard product of power series.
We provide a purely algebraic proof, which thus works over any field of characteristic zero.
Then we deduce a corresponding truncated version modulo a prime.
This congruence is not new either, but our deducing it from its infinite analogue is in line with the spirit of this paper.

In Section~\ref{sec:applications} we show how our polynomial congruences of Section~\ref{sec:polynomials}
can be evaluated on special values for $x$ so as to obtain numerical congruences.
The main obstacle here is the need for evaluations modulo $p$ in a closed form of the finite polylogarithms involved.
\comm{check}
Such evaluations, some of which were obtained in~\cite[Section~4]{MatTau:polylog},
are available only for a limited set of values of the argument, and fewer so as the level increases.
Nevertheless, a number of such congruences can be obtained, and we provide a sample of some new ones.

\section{A basic example: central binomial coefficients}\label{sec:original}

The central binomial coefficients have a generating function which we recalled in
Equation~\eqref{eq:central_gf}.
We will see that the polynomial obtained by omitting all terms of degree at least $q$ from the generating function
admits an equally nice closed form when viewed modulo $p$.
The following crucial step is essentially the same which leads to a proof of Lucas' theorem on binomial coefficients modulo $p$.

\begin{lemma}\label{lemma:central_pol}
If $q$ is a power of an odd prime $p$ we have
\begin{equation*}
\sum_{k=0}^{\infty}\binom{2k}{k}x^k\equiv(1-4x)^{(q-1)/2}
\cdot\sum_{k=0}^{\infty}\binom{2k}{k}x^{kq}\pmod{p}
\end{equation*}
in the formal power series ring $\Z[[x]]$.
\end{lemma}

\begin{proof}
Recall from Equation~\eqref{eq:central_gf} that
$
\sum_{k=0}^\infty\binom{2k}{k}x^k
=
(1-4x)^{-1/2}
$.
Basic facts about binomial coefficients and Fermat's Little Theorem imply that $(1-4x)^q\equiv 1-(4x)^q\equiv 1-4x^q\pmod{p}$.
Therefore, noting that all binomial power series involved have integral coefficients, we have
\begin{equation*}%\label{eq:1-4x}
\begin{aligned}
(1-4x)^{-1/2}
&=
(1-4x)^{(q-1)/2}\bigl((1-4x)^q\bigr)^{-1/2}
\\&\equiv
(1-4x)^{(q-1)/2}(1-4x^q)^{-1/2}
\pmod{p},
\end{aligned}
\end{equation*}
and the desired conclusion follows.
\end{proof}

By looking at the terms of degree less than $q$ in the congruence of power series given in Lemma~\ref{lemma:central_pol} we deduce the following polynomial congruence.
Its special case $q=p$ appeared in~\cite[p.~467]{Bacher-Chapman}, where the authors, however, suggested a proof by direct calculation.

\begin{theorem}\label{thm:central_pol}
If $q$ is a power of an odd prime $p$ we have
\begin{equation*}
\sum_{k=0}^{q-1}\binom{2k}{k}x^k\equiv(1-4x)^{(q-1)/2}\pmod{p}.
\end{equation*}
\end{theorem}

The exclusion of $p=2$ here is harmless as the central binomial coefficients are all even for $k>0$.
In a formula in $\Z[[x]]$, we have
$\sum_{k=0}^{\infty}\binom{2k}{k}x^k\equiv 1\pmod{2}$.

If desired, Lemma~\ref{lemma:central_pol} readily provides a closed form evaluation modulo $p$ for the sum
$\sum_{k=0}^{rq-1}\binom{2k}{k}x^k$,
for any given positive integer $r$, such as
\[
\sum_{k=0}^{2q-1}\binom{2k}{k}x^k\equiv(1-4x)^{(q-1)/2}(1+2x^q)\pmod{p}.
\]
In the sequel we will disregard such straightforward extensions
and focus on the most natural range $0\le k<q$ for similar sums.

Variations on the congruence of Theorem~\ref{thm:central_pol} where the central binomial coefficients
are multiplied by fixed powers of $k$ are easily obtained by the familiar device of repeated application of the operator $xD$,
that is, differentiation followed by multiplication by $x$.
As an example, from Theorem~\ref{thm:central_pol} we obtain
\begin{equation}\label{eq:central_pol_k}
\sum_{k=1}^{q-1}k\binom{2k}{k}x^k
\equiv
xD(1-4x)^{(q-1)/2}
\equiv
2x(1-4x)^{(q-3)/2}
\pmod{p}
\end{equation}
for $p$ odd.

Before we continue we introduce a convenient piece of notation about congruences between power series.
By a congruence with respect to a modulus $(x^s,p)$ between power series we mean that the polynomials obtained from them
by discarding all terms of degree $s$ or higher have integral coefficients and are congruent modulo $p$.
When both sides have integral coefficients this amounts to equality of their images in the quotient ring $\Z[[x]]/(x^s,p)$.
However, it will be convenient to allow ourselves greater flexibility and only require that the coefficients
of powers of $x$ of exponents less than $s$ are integers;
this does not admit a natural interpretation in terms of quotient rings.

Our next result concerns a polynomial version modulo $p$ of the generating function of the Catalan numbers,
\begin{equation}\label{eq:Catalan_gf}
\sum_{k=0}^\infty C_kx^{k+1}
=
\frac{1-\sqrt{1-4x}}{2}.
\end{equation}
This well-known generating function is usually written with both sides divided by $x$,
but the above formulation is more convenient for our present goals.
Equation~\eqref{eq:Catalan_gf} can be obtained by integrating Equation~\eqref{eq:central_gf} and adjusting the constant term.

\begin{theorem}\label{thm:Catalan_pol}
If $q$ is a power of an odd prime $p$ we have
\begin{equation*}%\label{eq:Catalan_pol}
\sum_{k=0}^{q-1}C_kx^{k+1}\equiv\frac{1-(1-4x)^{(q+1)/2}}{2}-x^q\pmod{p}.
\end{equation*}
\end{theorem}

As was the case for Theorem~\ref{thm:central_pol},
the excluded case $p=2$ in Theorem~\ref{thm:central_pol} can easily be dealt with directly, as
$\sum_{k=0}^{\infty}C_kx^k\equiv \sum_{i=0}^\infty x^{2^i-1}\pmod{2}$.
From now on we will disregard the case $p=2$ in our congruences as trivial and easily dealt with separately.

\begin{proof}
From Lemma~\ref{lemma:central_pol} we obtain
\[
(1-4x)^{-1/2}
\equiv
(1-4x)^{(q-1)/2}
+2x^q
\pmod{(x^{q+1},p)},
\]
whence
\[
(1-4x)^{1/2}
=
(1-4x)^{-1/2}(1-4x)
\equiv
(1-4x)^{(q+1)/2}
+2x^q
\pmod{(x^{q+1},p)}.
\]
Therefore, we have
\[
\sum_{k=0}^\infty C_kx^{k+1}
=
\frac{1-\sqrt{1-4x}}{2}
\equiv
\frac{1-(1-4x)^{(q+1)/2}}{2}-x^q\pmod{(x^{q+1},p)},
\]
which yields the desired conclusion.
\end{proof}

The congruences of Equation~\eqref{eq:sums_mod_p}, for odd $p$,
follow by evaluation at $x=1$ of the polynomials of Theorems~\ref{thm:central_pol} and~\ref{thm:Catalan_pol},
using the fact that
$(-3)^{(q-1)/2}
\equiv\tLegendre{q}{3}
%\equiv\bigl(\frac{q}{3}\bigr)
\pmod{p}$.
This is easily shown
either by using Jacobi symbols and Gauss' quadratic reciprocity law,
or by viewing $-3$ as the discriminant of the polynomial
$(x^3-1)/(x-1)=x^2+x+1$, and then evaluating on $-3$ the quadratic character of the finite field $\F_q$.

Generating functions of combinatorial sequences which have a similar form as that of the central binomial coefficients
admit a similar treatment.
We give only one example.
The {\em central trinomial coefficient} $T_k$ is the coefficient of $x^k$ in $(1+x+x^2)^k$.
It is well known and easy to prove that these numbers admit the generating function
\[
\sum_{k=0}^{\infty}T_kx^k
=\bigl((1-3x)(1+x)\bigr)^{-1/2}
=(1-2x-3x^2)^{-1/2}.
\]
Exactly as in the proof of Theorem~\ref{thm:central_pol} we find the following congruence.

\begin{theorem}\label{thm:central_pol_trinomial}
If $q$ is a power of an odd prime $p$ we have
\begin{equation*}%\label{eq:gf_cbc_q}
\sum_{k=0}^{q-1}T_kx^k\equiv(1-2x-3x^2)^{(q-1)/2}\pmod{p}.
\end{equation*}
\end{theorem}

For example, this shows that for any odd prime we have
$\sum_{k=0}^{p-1}T_k\equiv (-1)^{(p-1)/2}\pmod{p}$,
and
$\sum_{k=0}^{p-1}(-1)^kT_k\equiv 0\pmod{p}$.

\section{A variation: shifted central binomial coefficients}\label{sec:shifted}

In this section we consider {\em shifted} variants $\binom{2k}{k+d}$ of the central binomial coefficients,
as done in~\cite{PanSun} and in various papers which followed.
We prefer to use $\binom{2k+e}{k}$ instead, which takes care of
sums of binomial coefficients $\binom{2k-1}{k+d}$ as well as of $\binom{2k}{k+d}$ at the same time,
because the generating function for the corresponding series appears to be better known.
For any nonnegative integer $e$ we have
\begin{equation}\label{eq:central_shifted_gf}
\sum_{k\ge 0}
\binom{2k+e}{k}x^{k}
=
\frac{1}{\sqrt{1-4x}}
\left(\frac{1-\sqrt{1-4x}}{2x}\right)^{e},
\end{equation}
see~\cite[Equation~(2.47)]{Wilf}.
We will now devise a truncated version modulo $p$ of this equation, over the range $0\le k<q$.

Because
$\binom{2k+e+q}{k}\equiv\binom{2k+e}{k}$
for $0\le k<q$ and $e\ge 0$, according to Lucas' theorem,
we may and will assume $0\le e<q$.
Again because of Lucas' theorem we have
$\binom{2k+e}{k}\equiv 0\pmod{p}$
for $(q-e)/2\le k<q-e$,
and also for $q-e/2\le k<q$.
Hence within the range $0\le k<q$, the binomial coefficient can only be nonzero modulo $p$
on the two subintervals
$0\le k<(q-e)/2$ and $q-e\le k<q-e/2$.
It is convenient to separate the contributions of those two ranges in our polynomial congruence.

\begin{theorem}\label{thm:central_shifted_pol}
Let $q$ be a power of an odd prime $p$, and let $0\le e<q$.
In the polynomial ring $\Z[\beta]$, setting $x=\beta(1-\beta)$ and $\alpha=1-\beta$, we have
\begin{equation*}
\sum_{0\le k<(q-e)/2}\binom{2k+e}{k}x^k
\equiv
\frac{\alpha^{q-e}-\beta^{q-e}}{\alpha-\beta}
\pmod{p},
\end{equation*}
and
\begin{equation*}
\sum_{0\le k<q}\binom{2k+e}{k}x^k
\equiv
\frac{\alpha^{2q-e}-\beta^{2q-e}}{\alpha-\beta}
\pmod{p}.
\end{equation*}
\end{theorem}

At face value, because of the presence of a denominator
the right-hand sides of the congruences in Theorem~\ref{thm:central_shifted_pol} appear to belong only to the power series ring $\Z[[x]]$,
but they are actually polynomials after cancellation takes place.
Both congruences extend Theorem~\ref{thm:central_pol}, which is the special case $e=0$.

The rationale for introducing a new indeterminate $\beta$ will be clearer in Section~\ref{sec:series}, see Equation~\eqref{eq:beta}.
It is essentially a device to have polynomial congruences in $\beta$ rather than involving power series in $x$.
In terms of $x$, the indeterminate $\beta$ can be taken as the generating function of the Catalan numbers in the form given in Equation~\eqref{eq:Catalan_gf}, that is,
$\beta=(1-\sqrt{1-4x})/2$, an element of the power series ring $\Z[[x]]$.

\begin{proof}
In terms of $\beta$, Equation~\eqref{eq:central_shifted_gf} reads
\[
\sum_{k=0}^\infty\binom{2k+e}{k}x^k
=
\frac{1}{(1-2\beta)(1-\beta)^e},
\]
which takes place in the power series ring $\Q[[\beta]]$, with $x=\beta(1-\beta)$.
However, because all coefficients are integers it actually takes place in $\Z[[x]]$.
After multiplying both sides by $(1-\beta)^e$ and then by $(1-2\beta)^q\equiv 1\pmod{(\beta^q,p)}$ we obtain
\begin{equation}\label{eq:central_shifted_pol_tmp}
(1-\beta)^e
\sum_{0\le k<q}\binom{2k+e}{k}x^k
\equiv
(1-2\beta)^{q-1}
\pmod{(\beta^q,p)}.
\end{equation}
Because the binomial coefficients involved in the left-hand side vanish modulo $p$ for $q-e/2\le k<q$ the left-hand side is a polynomial of degree less than $2q$
when viewed modulo $p$, and so the congruence requires the double modulus $(\beta^q,p)$ to be valid.
However, once we restrict the summation range in Equation~\eqref{eq:central_shifted_pol_tmp}
to those two subintervals where the binomial coefficients may possibly not vanish modulo $p$,
for the higher subinterval we have
\[
(1-\beta)^e\sum_{q-e\le k<q-e/2}\binom{2k+e}{k}x^{k}
\equiv
\beta^{q-e}
\sum_{q-e\le k<q-e/2}\binom{2k+e}{k}x^{k-q+e}
\pmod{(\beta^q,p)},
\]
because $(1-\beta)^ex^{q-e}=(1-\beta)^q\beta^{q-e}\equiv\beta^{q-e}\pmod{(\beta^q,p)}$.
Making this replacement Equation~\eqref{eq:central_shifted_pol_tmp} turns it into
\begin{multline*}%\label{eq:central_shifted_pol}
(1-\beta)^e
\sum_{0\le k<(q-e)/2}\binom{2k+e}{k}x^k
+
\beta^{q-e}
\sum_{q-e\le k<q-e/2}\binom{2k+e}{k}x^{k-q+e}
\\
\equiv
(1-2\beta)^{q-1}
\pmod{(\beta^q,p)}.
\end{multline*}
However, the left-hand side of this congruence is now a polynomial of degree less than $q$ (in the indeterminate $\beta$).
Because so is the right-hand side, the congruence actually holds modulo $p$.

Now consider the above congruence and the one obtained from it by
interchanging the roles of $\beta$ and $\alpha=1-\beta$.
Taking suitable linear combinations of them
we obtain the first congruence of the theorem, as well as
\[
\sum_{q-e\le k<q-e/2}\binom{2k+e}{k}x^{k-q+e}
\equiv
\frac{\alpha^e-\beta^e}{\alpha-\beta}
\pmod{p}.
\]
Multiplying both sides of this congruence by $x^{q-e}=\alpha^{q-e}\beta^{q-e}$
and adding it to the first congruence of the theorem we obtain the second congruence of the theorem.
\end{proof}

Note the similarity of the congruences of Theorem~\ref{thm:central_shifted_pol} with
Equation~\eqref{eq:central_shifted_gf}.
However, the congruences are (necessarily) invariant under interchanging $\beta$ with $\alpha=1-\beta$,
and this fact has no counterpart in Equation~\eqref{eq:central_shifted_gf}.

Setting $e=2d$ in Theorem~\ref{thm:central_shifted_pol},
appropriately shifting the summation range, and considering the vanishing modulo $p$ of the binomial coefficients involved over part of the range,
we obtain
\begin{equation*}
\sum_{0\le k<q}\binom{2k}{k-d}x^{k-d}
\equiv
\frac{\alpha^{2q-2d}-\beta^{2q-2d}}{\alpha-\beta}
\pmod{p}
\end{equation*}
for $0\le d<q$.
More precisely, this alternate formulation follows directly as described for $0\le d<q/2$,
and after a simple manipulation for $q/2<d<q$.
By specializing $\beta$ to be a complex primitive sixth root of unity, whence $x=1$, we obtain Equation~\eqref{eq:2k_k+d} of the Introduction.
Variations such as
\begin{equation*}
\sum_{0\le k<q}k\binom{2k}{k-d}x^{k-d}
\equiv
\frac{
2\bigl(\alpha^{2q-2d}-\beta^{2q-2d}\bigr)x
}{(\alpha-\beta)^3}
+
\frac{
d\bigl(\alpha^{2q-2d}+\beta^{2q-2d}\bigr)
}{(\alpha-\beta)^2}
\pmod{p}
\end{equation*}
can be obtained by a suitable application of the differential operator
$\mathrm{d}/\mathrm{d}x=\bigl(1/(1-2\beta)\bigr)\cdot\mathrm{d}/\mathrm{d}\beta$.

\section{A sequence transform and its modular version}\label{sec:transform}

In this section we consider an involutory transform for sequences, which will be needed later in the paper,
and show that a truncated version modulo $p$
can be deduced through a similar method as employed in the previous section.

We start by recalling the more well-known {\em binomial transform}.
Given a sequence $(a_n)_{n=0}^{\infty}$ of elements of a field $F$ (or of any ring, for that matter),
its {\em binomial transform} is the sequence $(b_n)$ defined by
$b_n=\sum_{k=0}^{n}(-1)^k\binom{n}{k} a_k$,
which is more conveniently written as
$b_n=\sum_{k=0}^{\infty}(-1)^k\binom{n}{k} a_k$.
With our choice of signs (not shared by all authors) the binomial transform is {\em involutory,}
meaning that it coincides with the inverse transform, and hence
$a_n=\sum_{k=0}^{\infty}(-1)^k\binom{n}{k} b_k$.
An easy way to see this is noting that the exponential generating functions
$\tilde A(x)=\sum_{k=0}^{\infty}a_kx^k/k!$ and
$\tilde B(x)=\sum_{k=0}^{\infty}b_kx^k/k!$
are related by
$\tilde B(x)=\exp(-x)\cdot\tilde A(-x)$.

Here we are interested in the related transform
\begin{equation}\label{eq:transform}
\sum_{k=0}^{\infty}\binom{2k}{k} a_k x^k
=\frac{1}{\sqrt{1-4x}}\sum_{k=0}^{\infty}
\binom{2k}{k}b_k \left(\frac{-x}{1-4x}\right)^k,
\end{equation}
where $(a_n)$ and $(b_n)$ are connected by the binomial transform.
Writing $\binom{2k}{k}=(-4)^k\binom{-1/2}{k}$,
Equation~\eqref{eq:transform} follows from the following more general fact by evaluating at $y=-1/2$.

\begin{prop}\label{prop:transform_general}
Let $F$ be a field of characteristic zero.
In the ring $\bigl(F[y]\bigr)[[x]]$ we have
\begin{equation}\label{eq:transform_general}
\sum_{k=0}^{\infty}\binom{y}{k} a_k x^k
=(1+x)^y\sum_{k=0}^{\infty}
\binom{y}{k}b_k \left(\frac{x}{1+x}\right)^k,
\end{equation}
where $a_k\in F$ and $b_n=\sum_{k=0}^{\infty}(-1)^k\binom{n}{k} a_k$
for all $n\ge 0$.
\end{prop}

By definition we have $(1+x)^y=\sum_{k=0}^{\infty}\binom{y}{k}x^k$, which belongs to
$\bigl(\Q[y]\bigr)[[x]]\subseteq\bigl(F[y]\bigr)[[x]]$.
The expected property
$(1+x)^{y_1+y_2}=(1+x)^{y_1}\cdot(1+x)^{y_2}$
holds, and takes place in $\bigl(\Q[y_1,y_2]\bigr)[[x]]$.
For sequences of complex numbers, and with complex parameter in place of the indeterminate $y$ (and with a different sign choice)
Proposition~\ref{prop:transform_general} is~\cite[Proposition~5]{Boy},
but in essence the result traces back to Euler, see~\cite[Equation~(1.20)]{Norlund}.

To be precise, the special version of Proposition~\ref{prop:transform_general} proved in~\cite{Boy}
also requires that $\sum_{k=0}^{\infty}a_k x^k$ has a positive radius of convergence.
This is because that proof amounts to viewing the left-hand side of Equation~\eqref{eq:transform_general}
as the Hadamard (that is, termwise) product of the series
$\sum_{k=0}^{\infty}\binom{y}{k} x^k$ and
$\sum_{k=0}^{\infty}a_k x^k$
and applying a familiar integral formula for it (see~\cite[Chapter~1, Exercise~30]{Com}).
Our purely algebraic approach below bypasses the analytic tool and allows us to attain the greater generality of Proposition~\ref{prop:transform_general}.
For example, one can obtain {\em parametrized} versions by taking $F$ to be a function field.

\begin{proof}[Proof of Proposition~\ref{prop:transform_general}]
Fix an index $N\ge 0$.
If we change all $a_k$ with $k>N$ into zeroes, then
$b_k=\sum_{j=0}^{\infty}(-1)^j\binom{k}{j} a_j$ change accordingly, but $b_k$ remains unchanged for $k\le N$.
The coefficient of $x^N$ in either side of Equation~\eqref{eq:transform_general} is unaffected by this change.
Thus, in proving Equation~\eqref{eq:transform_general} we may assume that $a_k=0$ for $k>N$.
This artifice will save us the trouble of working with bilateral Laurent series, which of course do not form a ring.

Introduce a further indeterminate $z$, so that we have
\[
b_k=\sum_{j=0}^{k}(-1)^j\binom{k}{j} a_j
=[z^0]\biggl((-1)^k(1-z)^k\sum_{j=0}^{N}a_j z^{-j}\biggr)
\]
in the ring $R((z))$ of formal Laurent series over the integral domain $R=\bigl(F[y]\bigr)[[x]]$.
Hence
\begin{align*}
\sum_{k=0}^{\infty}
\binom{y}{k}b_k \left(\frac{-x}{1+x}\right)^k
&=
\sum_{k=0}^{\infty}
\binom{y}{k}\left(\frac{-x}{1+x}\right)^k
\cdot[z^0]
\biggl((1-z)^k\sum_{j=0}^{N}a_j z^{-j}\biggr)
\\&=
[z^0]
\left(
\sum_{k=0}^{\infty}
\binom{y}{k}\left(\frac{x(z-1)}{(1+x)}\right)^k
\sum_{j=0}^{N}a_j z^{-j}
\right)
\\&=
[z^0]
\left(
\left(1+\frac{x(z-1)}{(1+x)}\right)^y
\sum_{j=0}^{N}a_j z^{-j}
\right)
\\&=
(1+x)^{-y}\cdot
[z^0]
\left(
(1+xz)^y
\sum_{j=0}^{N}a_j z^{-j}
\right)
\\&=
(1+x)^{-y}\cdot
\sum_{k=0}^{N}\binom{y}{k} a_k x^k.
\end{align*}
The conclusion follows upon multiplication by $(1+x)^y$.
\end{proof}

Now we show how similar arguments as in Section~\ref{sec:original} allow one to deduce from Equation~\eqref{eq:transform}
a truncated version modulo a prime.
For simplicity of exposition we work with sequences of rational numbers,
but it will be clear that it extends to more general contexts if needed, such as number fields.
Thus, let $q$ be a power of an odd prime $p$, let $a_0,a_1,\ldots,a_{q-1}$ be $p$-integral rational numbers,
and set $b_n=\sum_{k=0}^{n}(-1)^k\binom{n}{k} a_k$
for $0\le n<q-1$.
Then in $\Q[x]$ we have the congruence
\begin{equation}\label{eq:transform_truncated}
\sum_{k=0}^{q-1}\binom{2k}{k} a_k x^k
\equiv(1-4x)^{(q-1)/2}\sum_{k=0}^{q-1}
\binom{2k}{k}b_k \left(\frac{-x}{1-4x}\right)^k
\pmod{p}.
\end{equation}
Equation~\eqref{eq:transform_truncated} is not hard to establish directly, as in~\cite[Section~3]{Tauraso:Fibonacci},
but deriving it from Equation~\eqref{eq:transform} is more in line with the spirit of this paper.
To do that, extend the finite sequence $(a_k)$ to an infinite sequence by setting $a_k=0$ for $k\ge q$,
and let $(b_k)$ be the binomial transform of $(a_k)$.
Hence Equation~\eqref{eq:transform} holds and, consequently,
\[
\sum_{k=0}^{q-1}\binom{2k}{k} a_k x^k
\equiv\frac{1}{\sqrt{1-4x}}\sum_{k=0}^{q-1}
\binom{2k}{k}b_k \left(\frac{-x}{1-4x}\right)^k
\pmod{(x^q,p)}.
\]
Using
$(1-4x)^{-1/2}
\equiv
(1-4x)^{(q-1)/2}(1-4x^q)^{-1/2}
\pmod{p}$
as in the proof of Lemma~\ref{lemma:central_pol}
we obtain that Equation~\eqref{eq:transform_truncated} holds modulo $(x^q,p)$.
However, because $\binom{2k}{k}\equiv 0\pmod{p}$ for $q<2k<2q$ we may restrict the summation range in the right-hand side
of Equation~\eqref{eq:transform_truncated} to $0\le(q-1)/2$, thus making both sides polynomials of degree less than $p$,
whence Equation~\eqref{eq:transform_truncated} holds modulo $p$, as desired.

\section{Some generating series evaluated in terms of polylogarithms}\label{sec:series}

In this section we recall a few known generating series involving polylogarithms, and introduce some new ones,
towards the goal of truncating them and obtaining polynomial congruences modulo a prime in Section~\ref{sec:polynomials}.
All our series will be generating series of sequences of the general form
$\binom{2k}{k}a_k$
or
$C_ka_k$,
where $a_k$ might be $1/k^d$, or a generalized harmonic number
$H_k^{(d)}=\sum_{j=1}^{k}1/j^d$
(where the ordinary harmonic numbers are $H_k=H_k^{(1)}$),
or possibly a linear combination of products of them.
For the sake of classification we will informally call {\em the level} of such a generating series
the highest power of $k$ which occurs in the denominator of $a_k$ once expanded.
Some such generating series can be evaluated in closed form in terms of polylogarithmic series
\begin{equation}\label{eq:Li_d}
\Li_d(x)=\sum_{k=1}^{\infty}\frac{x^k}{k^d}
\end{equation}
with $d$ equal to the level of the series,
and the ordinary logarithmic series.
However, we will systematically write $\Li_1(x)$ in place of the equivalent notation $-\log(1-x)$
in our formulas.
This will be more natural in view of deducing congruences modulo a prime in Section~\ref{sec:polynomials}.

We may and will work in a formal setting, viewing all series involved as formal power series, say in $\Q[[x]]$ or $\C[[x]]$,
as subsequent evaluations inside the disk of convergence pose no challenge when desired.
Thus, polylogarithms $\Li_d(x)$ will be simply defined by Equation~\eqref{eq:Li_d} for any integer $d$,
and no issue of analytic continuation will arise.
Note that $\Li_0(x)=x/(1-x)$.
Differentiation and integration will be done on a formal level in $\Q[[x]]$.
In particular, note that
\[
x\cdot\frac{\mathrm{d}}{\mathrm{d}x}\Li_d(x)=\Li_{d-1}(x).
\]

Because (a slight variant of) the generating function of the Catalan numbers will occur repeatedly
we conveniently assign a name to it, namely,
\begin{equation}\label{eq:beta}
\beta:=
\sum_{k=0}^\infty C_kx^{k+1}
=
\frac{1-\sqrt{1-4x}}{2}.
\end{equation}
It satifies $\beta(1-\beta)=x$.
The generating series of the central binomial coefficients can also be expressed in terms of $\beta$, namely,
$\sum_{k=0}^{\infty}\binom{2k}{k}x^k
%=1/\sqrt{1-4x}
=1/(1-2\beta)
=\mathrm{d}\beta/\mathrm{d}x$.
According to our terminology these two series are the generating series of level zero.
Series with general term $C_ka_k$ can be obtained by integration from the corresponding ones with term $\binom{2k}{k}a_k$,
and so our initial focus will be on the latter.
Integration of the corresponding closed forms will present no obstacle in the cases of our concern.
Our notation for indefinite integration will assume that the arbitrary constant involved will be suitably adjusted,
allowing us to write $\int(1-2\beta)^{-1}\mathrm{d}x=\beta$, for example.

From each series of level $d$ one can obtain another series of the same level by an application of the involutory transform described by Equation~\eqref{eq:transform}.
Going from a series of level $d$ to a series of level $d+1$ may be achieved by integration.
Thus, the easiest generating series of level one is obtained by integrating the generating series of the central binomial coefficients divided by $x$.
As pointed out in~\cite[Equation~(6)]{Lehmer:central_binomial}, this yields
\begin{equation}\label{B4}
\sum_{k=1}^{\infty}\binom{2k}{k}\frac{x^k}{k}=
-2\log(1-\beta)=2\Li_1\left(\beta\right).
\end{equation}
An application of the transform of Equation~\eqref{eq:transform} turns this into another series of level one,
obtained by Boyadzhiev in~\cite[Theorem~1]{Boy}.
In our notation in terms of $\beta$, that result concisely reads
\begin{equation}\label{B1}
\sum_{k=1}^{\infty}\binom{2k}{k}H_kx^k=
\frac{-2}{1-2\beta}\Li_1\left(\frac{\beta}{2\beta-1}\right)
=
2\frac{\Li_1(2\beta)-\Li_1(\beta)}{1-2\beta}.
\end{equation}
These are just two of several equivalent expressions for this series,
due to the functional equation $\log(1+x+y+xy)=\log(1+x)\log(1+y)$.
The former expression may appear more convenient due to the single occurrence of $\Li_1$,
but the latter will soon prove to be more compatible with analogues of higher level,
and more amenable to reduction modulo a prime in the next section.

Our collection of series of level one is completed with two corresponding series involving the Catalan numbers.
The analogue of Equation~\eqref{B4}, which reads
\begin{equation}\label{M1}
\sum_{k=1}^{\infty}C_k\frac{x^{k+1}}{k}
=
\sum_{k=1}^{\infty}\binom{2k}{k}\frac{x^{k+1}}{k}
-
\sum_{k=1}^{\infty}C_k x^{k+1}
=
2x\Li_1(\beta)-\beta+x,
\end{equation}
is obtained from Equations~\eqref{B4} and~\eqref{eq:beta}
noting that
$1/\bigl(k(k+1)\bigr)=1/k-1/(k+1)$.
Alternatively, Equation~\eqref{M1} can be found by integrating Equation~\eqref{B4}.
Finally, integrating Equation~\eqref{B1} in either form we obtain
\begin{equation}\label{B5}
\begin{aligned}
\sum_{k=1}^{\infty}C_kH_kx^{k+1}
&=
\Li_1(\beta)+(1-2\beta)\Li_1\left(\frac{\beta}{2\beta-1}\right)
\\&=
-(1-2\beta)\Li_1(2\beta)+2(1-\beta)\Li_1(\beta),
\end{aligned}
\end{equation}
which is equivalent to~\cite[Corollary~2]{Boy}.
We should mention that Equation~\eqref{B5} can also be obtained from Equation~\eqref{M1} through an application of the transform expressed by Equation~\eqref{eq:transform}.

Boyadzhiev also produced closed forms for some generating series of level two
in~\cite[Proposition~6 and Corollary~7]{Boy},
which we can state in the more compact forms
\begin{equation}\label{B3}
\sum_{k=1}^{\infty}\binom{2k}{k}\frac{x^k}{k^2}
=2\Li_2\left(\beta\right)
-\Li_1(\beta)^2,
\end{equation}
and
\begin{equation}\label{B2}
\sum_{k=1}^{\infty}\binom{2k}{k}\frac{H_k}{k}x^k=
-2\Li_2\left(\frac{\beta}{2\beta-1}\right)
-\Li_1\left(\frac{\beta}{2\beta-1}\right)^2.
\end{equation}
They are obtained by integrating Equations~\eqref{B4} and~\eqref{B1} after division by $x$.
Companion generating series with the Catalan numbers may be readily obtained using
$1/\bigl(k(k+1)\bigr)=1/k-1/(k+1)$.

In the next result we contribute two further generating series of level two,
which involve the generalized harmonic numbers $H_k^{(2)}$.

\begin{theorem}
We have
\begin{equation}\label{S2}
\sum_{k=1}^{\infty}\binom{2k}{k}H_k^{(2)} x^k
=
\frac{2\Li_2\left(\beta\right)
+\Li_1(\beta)^2}{1-2\beta},
\end{equation}
and
\begin{equation}\label{S3}
\sum_{k=1}^{\infty}C_kH_k^{(2)}x^{k+1}=
2\beta\,\Li_2\left(\beta\right)
-(1-\beta)\Li_1(\beta)^2.
\end{equation}
\end{theorem}

\begin{proof}
Note that
\[
\sum_{j=1}^{k}\binom{k}{j}(-1)^j H_j^{(2)}=-\frac{H_k}{k},
\]
according to an easy summation by parts.
Now Equation~\eqref{S2} can be obtained from
Equation~\eqref{B2} through the transform of Equation~\eqref{eq:transform}, as follows:
\begin{equation*}
\begin{split}
\sum_{k=1}^{\infty}\binom{2k}{k}H_k^{(2)} x^k
&=\frac{1}{1-2\beta}
\sum_{k=1}^{\infty}\binom{2k}{k}\left(\frac{-x}{1-4x}\right)^k
\sum_{j=1}^{k}\binom{k}{j}(-1)^j H_j^{(2)}\\
&=-\frac{1}{1-2\beta}
\sum_{k=1}^{\infty}\binom{2k}{k}\frac{H_k}{k}\left(\frac{-x}{1-4x}\right)^k\\
&=
\frac{2\Li_2\left(\beta\right)+\Li_1(\beta)^2}{1-2\beta}.
\end{split}
\end{equation*}

Equation~\eqref{S3} can be obtained from Equation~\eqref{S2} through integration.
In fact, integration by parts yields
\[
\int\frac{\Li_2(\beta)}{1-2\beta}\,\mathrm{d}x
=
\int\Li_2(\beta)\mathrm{d}\beta
=
\beta\Li_2(\beta)+(1-\beta)\Li_1(\beta)-\beta
\]
and
\[
\int\frac{\Li_1(\beta)^2}{1-2\beta}\,\mathrm{d}x
=
\int\Li_1(\beta)^2\,\mathrm{d}\beta
=
(\beta-1)\Li_1(\beta)^2+2(\beta-1)\Li_1(\beta)+2\beta,
\]
whence Equation~\eqref{S3} follows.
\end{proof}

Finally, we evaluate one generating series of level three in closed form, where two suitable summands are combined.

\begin{theorem}
We have
\begin{equation}\label{S4}
\sum_{k=1}^{\infty}
\binom{2k}{k}\left(\frac{H_k^{(2)}}{k}+\frac{1}{k^3}\right)x^k=
4\,\Li_3\left(\beta\right)
+\frac{2}{3}\Li_1(\beta)^3.
\end{equation}
\end{theorem}

\begin{proof}
We start with noting that for $d\geq 1$ we have
\[
\frac{\mathrm{d}}{\mathrm{d}x}\Li_d(\beta)=
\frac{\Li_{d-1}(\beta)}{\beta(1-2\beta)},
\quad\text{and}\quad
\frac{\mathrm{d}}{\mathrm{d}x}\Li_1(\beta)^d=
\frac{d\Li_1(\beta)^{d-1}}{(1-\beta)(1-2\beta)}.
\]
Adding up the left-hand sides of Equations~\eqref{S2} and~\eqref{B3},
dividing by $x$ and integrating, yields
\begin{align*}
\sum_{k=1}^{\infty}
\binom{2k}{k}\left(\frac{H_k^{(2)}}{k}+\frac{1}{k^3}\right)x^k&=
\int \left(\frac{2\Li_2\left(\beta\right)
+\Li_1(\beta)^2}{x(1-2\beta)}+
 \frac{2\Li_2(\beta)
-\Li_1(\beta)^2}{x}\right)\,\mathrm{d}x\\
&=4\int \frac{\Li_2(\beta)}{\beta(1-2\beta)}\,\mathrm{d}x
 +2\int \frac{\Li_1(\beta)^2}{(1-\beta)(1-2\beta)}\,\mathrm{d}x\\
&=4\,\Li_3\left(\beta\right)
+\frac{2}{3}\Li_1(\beta)^3,
\end{align*}
as claimed.
\end{proof}

\section{Congruences for finite polylogarithms}\label{sec:polylog}

In Section~\ref{sec:polynomials} we will obtain congruences for finite sums modulo a prime $p$ from the generating series found in Section~\ref{sec:series}.
Because each of those involves some polylogarithm $\Li_d$, we give here a brief introduction to their finite analogues $\pounds_d$.
For a fixed prime $p$, the {\em finite polylogarithms} can be defined by truncating the polylogarithmic series just before the term of degree $p$, namely,
\begin{equation}\label{eq:L_d}
\pounds_d(x)=\sum_{k=1}^{p-1}\frac{x^k}{k^d}.
\end{equation}
Although they will be mostly viewed modulo $p$, hence over the field $\F_p$,
there is an advantage in having them defined as polynomials with rational coefficients.
Thus, we have
$\pounds_d(x)\equiv\Li_d(x)\pmod{x^p}$
in the power series ring $\Q[[x]]$ and, in particular,
$\pounds_1(x)\equiv -\log(1-x)\pmod{x^p}$.
Note also that $\pounds_0(x)=(x-x^p)/(1-x)$.

When $p=2$ we have $\pounds_d(x)=x$ for all $d$, which is not very interesting.
Nor is the behaviour of the central binomial coefficients $\binom{2k}{k}$ or the Catalan numbers $C_k$
when viewed modulo $2$, as we already pointed out right after stating Theorems~\ref{thm:central_pol} and~\ref{thm:Catalan_pol}.
Thus, we set the blanket assumption $p>2$ in what follows
and leave to the interested reader the task of checking what remains true or fails in that case.
More generally, because
$\pounds_{d+p-1}(x)\equiv\pounds_d(x)\pmod{p}$,
an assumption $0<d<p$ would not be too demanding in most places, but we need not require that from the outset.

We will need the congruence
\begin{equation}\label{eq:Mirimanoff}
\pounds_1(x)^d\equiv(-1)^{d-1}d!\cdot\pounds_d(1-x)
\pmod{(x^p,p)},
\end{equation}
valid for $0<d<p-1$,
which is a weaker version of a more precise congruence, modulo $(x^{p+1},p)$,
tracing back to Mirimanoff and involving a Bernoulli number.
A proof of that, with a further discussion,
can be found in~\cite[Lemma~3.2]{MatTau:polylog}.

For small values of $d$, Equation~\eqref{eq:Mirimanoff} can be refined to a congruence modulo $p$, namely,
\begin{align}
\label{eq:L^1}
\pounds_1(x)
&\equiv
\pounds_1(1-x)\pmod{p},
\\\label{eq:L^2}
\pounds_1(x)^2/2
&\equiv -x^p\pounds_2(x)-(1-x^p)\pounds_2(1-x) \pmod{p},
\\\label{eq:L^3}
\pounds_1(x)^3/6
&\equiv
x^p\pounds_3(x)+(1-x^p)\pounds_3(1-x)+x^{2p}(1-x^p)\pounds_3(1-1/x)
\\&\notag
\quad
+(2/3)x^p(1-x^p)\pounds_3(-1) \pmod{p};
\end{align}
the second congruence clearly requires $p>2$, and the third one $p>3$.
As explained in~\cite[Section~3]{MatTau:polylog}, these congruences have been recently rediscovered by several authors,
but they were already known to Mirimanoff~\cite[p.~61]{Mirimanoff}.
Equation~\eqref{eq:L^1} is a plain consequence of
\begin{equation}\label{eq:Q}
\pounds_1(x)\equiv\frac{-x^{p}-(1-x)^p+1}{p}\pmod{p},
\end{equation}
which is easily proved by expanding $(1-x)^p$ and using the fact that
$\binom{p}{k}=\frac{p}{k}\binom{p-1}{k-1}\equiv(-1)^{k-1}p\pmod{p^2}$ for $0<k<p$.
The invariance of $\pounds_1(x)$ under the substitution $x\mapsto 1-x$, due to Equation~\eqref{eq:L^1},
combines with invariance under another involutory transformation, expressed by the obvious congruence
$\pounds_1(x)\equiv -x^p\pounds_1(1/x)\pmod{p}$,
to give invariance under a certain (celebrated) group $G$ of transformations, isomorphic with the symmetric group on three letters.
In~\cite[Section~3]{MatTau:polylog} we showed how this symmetry group allows one to lift a proof of
Equations~\eqref{eq:L^2} and~\eqref{eq:L^3}
from the fact that they hold modulo $(x^{p+1},p)$
according to Mirimanoff's slightly more precise version of Equation~\eqref{eq:Mirimanoff}
given in~\cite[Lemma~3.2]{MatTau:polylog}.

By suitably composing the two involutory transformations of $\pounds_1(x)$ given above we find
\[
\pounds_1(x)
\equiv\pounds_1(1-x)
\equiv (x-1)^p\pounds_1\left(\frac{1}{1-x}\right)
\equiv (x-1)^p\pounds_1\left(\frac{x}{x-1}\right)
\pmod{p},
\]
which corresponds to invariance under the third involution in the mentioned symmetry group $G$.
The congruence between the first and last expressions may be viewed as a version for $\pounds_1(x)$ of the property
\begin{equation}\label{eq:log-}
-\log(1-x)=\log\left(1-\frac{x}{x-1}\right)
\end{equation}
of the logarithmic series;
in fact, viewing both sides of the latter modulo $(x^p,p)$
yields the former modulo $(x^p,p)$.
It turns out that the fundamental two-variable functional equation for the logarithmic series,
which we may write in the form
\begin{equation}\label{eq:log}
-\log(1-x)+\log(1-y)=\log\left(\frac{1-y}{1-x}\right)
\end{equation}
and has Equation~\eqref{eq:log-} as its special case for $y=1$,
also has analogue for $\pounds_1(x)$, namely,
\begin{equation}\label{eq:4-term}
\pounds_1(x)-\pounds_1(y)
+x^p\pounds_1\left(\frac{y}{x}\right)
+(1-x)^p\pounds_1\left(\frac{1-y}{1-x}\right)
\equiv 0\pmod{p}.
\end{equation}
Because of the polynomials appearing as denominators this may be viewed in the ring of power series $\Z_p[[x,y]]$,
where $\Z_p$ denotes the ring of $p$-adic integers,
but it actually takes place in the polynomial ring $\Z_p[x,y]$ after cancellation.
It seems to have been first noticed by Kontsevich~\cite{Kontsevich} in a more rudimentary form,
where $\pounds_1$ is viewed as a map from $\F_p$ to itself,
in order to point out an analogy of $\pounds_1$ with the {\em binary entropy function} of information theory.
Kontsevich sketched a proof of his version of Equation~\eqref{eq:4-term},
which amounts to Equation~\eqref{eq:4-term} viewed modulo the ideal $(x^p-x,y^p-y,p)$ of $\Z_p[x,y]$,
by direct calculation after expanding the various $\pounds_1$ involved.
That argument can certainly be upgraded to a full proof of Equation~\eqref{eq:4-term},
which, however, would not be too illuminating.
Elbaz-Vincent and Gangl, in a much more general scheme (greatly motivated by~\cite{Kontsevich}) where
functional equations for $\pounds_d$ are deduced from functional equations for $\Li_{d+1}$,
gave another proof of Equation~\eqref{eq:4-term} in~\cite[Proposition~5.9]{EVG:polyanalogsI}.
While that proof is certainly clarifying, understanding it requires mastering a large part of their paper.

We now give two elementary and self-contained new proofs of Equation~\eqref{eq:4-term}.
The former is slightly longer but more in line with the spirit of this paper, as it emphasizes and exploits the connection with Equation~\eqref{eq:log},
the functional equation for the ordinary logarithm.

\begin{proof}[First proof of Equation~\eqref{eq:4-term}]
Viewing Equation~\eqref{eq:log} modulo the ideal $(x,y)^p$ of $\Q[[x,y]]$, rewriting in terms of $\pounds_1$, and then reducing modulo $p$,
we find
\[
\pounds_1(x)-\pounds_1(y)
+\pounds_1\left(1-\frac{1-y}{1-x}\right)
\equiv 0\pmod{((x,y)^p,p)}
\]
in $\Z_p[[x,y]]$.
Rewriting the last summand using $\pounds_1(1-z)\equiv\pounds_1(z)\pmod{p}$
we conclude that Equation~\eqref{eq:4-term} holds modulo $((x,y)^p,p)$.
Because the left-hand side of Equation~\eqref{eq:4-term} is a polynomial of degree not exceeding $p$,
the residual indeterminacy about terms of degree exactly $p$
could be abundantly resolved using the invariance of $\pounds_1$ under the group $G$ of transformations,
by argument similar to those in~\cite{MatTau:polylog}.
However, a simple alternative is checking that the terms of degree $p$ in Equation~\eqref{eq:4-term}
cancel out modulo $p$.
In fact, the homogeneous part of degree $p$ in the polynomial
\[
(1-x)^p\pounds_1\left(\frac{1-y}{1-x}\right)
=
\sum_{k=1}^{p-1}\frac{(1-x)^{p-k}(1-y)^k}{k}
\]
equals
$(-1)^p\sum_{k=1}^{p-1}x^{p-k}y^k/k
\equiv -x^p\pounds_1(y/x)\pmod{p}$.
\end{proof}

\begin{proof}[Second proof of Equation~\eqref{eq:4-term}]
Applying Equation~\eqref{eq:Q} to each of the four summands of Equation~\eqref{eq:4-term} yields, in particular,
\[
x^p\pounds_1\left(\frac{y}{x}\right)\equiv\frac{-y^{p}-(x-y)^p+x^p}{p}\pmod{p},
\]
which is a sort of homogeneous version of Equation~\eqref{eq:Q}, and
\[
(1-x)^p\pounds_1\left(\frac{1-y}{1-x}\right)\equiv\frac{-(1-y)^{p}-(y-x)^p+(1-x)^p}{p}\pmod{p}.
\]
Because $(-1)^p\equiv -1\pmod{p}$ and $(1-x)^p\equiv 1-x^p\pmod{p}$,
all summands in Equation~\eqref{eq:4-term} cancel out, as desired.
\end{proof}

\section{Obtaining polynomial congruences by truncation}\label{sec:polynomials}

After these preliminaries on finite polylogarithms we proceed with producing analogues with finite sums modulo $p$
from the generating series considered in Section~\ref{sec:series}.
We start by deducing three congruences involving central binomial coefficients, of level one, two, and three,
from the corresponding generating series, which are
Equations~\eqref{B4}, \eqref{B3}, and ~\eqref{S4}.
We collect them together because of the similarity of their right-hand sides.
Because of integers coprime with $p$ appearing as denominators we conveniently state those congruences, and all congruences to follow,
in the polynomial ring $\Z_p[\beta]$, where $\Z_p$ is the ring of $p$-adic integers and $x=\beta(1-\beta)$.

\begin{theorem}\label{thm:2kk}
Let $p>3$ be a prime.
The following congruences hold in the polynomial ring $\Z_p[\beta]$, where $x=\beta(1-\beta)$ and $\alpha=1-\beta$:
\begin{equation}\label{N1}
\sum_{k=1}^{p-1}\binom{2k}{k}\frac{x^k}{k}\equiv
\pounds_1(\alpha)+\pounds_1(\beta)\pmod{p},
\end{equation}
\begin{equation}\label{N2}
\sum_{k=1}^{p-1}\binom{2k}{k}\frac{x^k}{k^2}\equiv
2\pounds_2(\alpha)+2\pounds_2(\beta)\pmod{p},
\end{equation}
\begin{equation}\label{N3}
\sum_{k=1}^{p-1}\binom{2k}{k}\left(\frac{H_k^{(2)}}{k}+\frac{1}{k^3}\right)x^k\equiv
4\pounds_3(\alpha)+4\pounds_3(\beta)\pmod{p}.
\end{equation}
\end{theorem}

\begin{proof}
To prove Equation~\eqref{N1} we start with writing Equation~\eqref{B4} in the equivalent form
\[
\sum_{k=1}^{\infty}\binom{2k}{k}\frac{\bigl(\beta(1-\beta)\bigr)^k}{k}=
2\Li_1\left(\beta\right),
\]
in the power series ring $\Q[[\beta]]$.
Because $\Li_d(x)\equiv\pounds_d(x)\pmod{x^{p}}$ we have
\[
\sum_{k=1}^{(p-1)/2}\binom{2k}{k}\frac{\bigl(\beta(1-\beta)\bigr)^k}{k}
\equiv
2\pounds_1\left(\beta\right)
\pmod{\beta^p}
\]
in $\Q[[\beta]]$.
Now both sides are polynomials with $p$-integral rational coefficients, and hence can be viewed modulo $p$.
Thus, the same congruence holds modulo the ideal $(\beta^p,p)$ of $\Z_p[\beta]$.
Because both sides are polynomials of degree less than $p$, the congruence actually holds modulo $p$
(that is, modulo the ideal $(p)$ of $\Z_p[\beta]$).
This is equivalent to the more symmetric Equation~\eqref{N1} because $\pounds_1(1-\beta)\equiv\pounds_1(\beta)\pmod{p}$
according to Equation~\eqref{eq:L^1}.

To prove Equation~\eqref{N2} we proceed similarly, writing Equation~\eqref{B3} in the equivalent form
\[
\sum_{k=1}^{\infty}\binom{2k}{k}\frac{\bigl(\beta(1-\beta)\bigr)^k}{k^2}
=2\Li_2\left(\beta\right)
-\Li_1(\beta)^2,
\]
in $\Q[[\beta]]$.
Truncating the series and using Equation~\eqref{eq:Mirimanoff} we deduce the congruence
\[
\sum_{k=1}^{(p-1)/2}\binom{2k}{k}\frac{\bigl(\beta(1-\beta)\bigr)^k}{k^2}
\equiv
2\pounds_2\left(\beta\right)
+2\pounds_2\left(1-\beta\right)
\pmod{(\beta^p,p)}
\]
in $\Z_p[\beta]$.
Now both sides are polynomials of degree less than $p$, and
hence the congruence actually holds modulo $p$ as well,
which is the desired conclusion.

In an entirely similar manner one obtains Equation~\eqref{N3} from Equation~\eqref{S4},
again using Equation~\eqref{eq:Mirimanoff} to get rid of $\pounds_1(\beta)^3$
in favour of $\pounds_3(\alpha)$.
\end{proof}

According to need, the congruences of Theorem~\ref{thm:2kk} can be viewed as identities in the polynomial ring
$\F_p[\beta]$ or in the power series ring $\F_p[[x]]$, with $\beta$
the power series defined by Equation~\eqref{eq:beta}.
Taking the former viewpoint, the left-hand sides belong to the subfield $\F_p(x)$ of the field
$\F_p(\beta)$ of rational expressions fixed by the automorphism which interchanges $\beta$ and $\alpha=1-\beta$.
This invariance is emphasized by the form in which we have written the right-hand sides.
Viewing the congruences as between polynomials in $\beta$, rather than power series in $x$,
has the advantage of allowing evaluation on algebraic integers, as we will exemplify in Section~\ref{sec:applications}.

Next we present two further congruences of level one, which are finite analogues of Equations~\eqref{B1} and~\eqref{B5} found by Boyadzhev.

\begin{theorem}\label{thm:2kkHk}
Let $p$ be an odd prime.
The following congruences hold in the polynomial ring $\Z_p[\beta]$, where $x=\beta(1-\beta)$ and $\alpha=1-\beta$:
\begin{equation}\label{M3}
\sum_{k=1}^{p-1}\binom{2k}{k}H_kx^k
\equiv
-2(\alpha-\beta)^{p-1}\pounds_1\left(\frac{\beta}{\beta-\alpha}\right)
\pmod{p},
\end{equation}
\begin{equation}\label{CM3}
\sum_{k=1}^{p-1}C_kH_kx^{k+1}
\equiv
\pounds_1(\beta)
+(\alpha-\beta)^{p+1}\pounds_1\left(\frac{\beta}{\beta-\alpha}\right)
\pmod{p}.
\end{equation}
\end{theorem}

The right-hand sides of Equations~\eqref{M3} and~\eqref{CM3} are actually polynomials, once $\pounds_1$ is expanded into a sum
and cancellation with the power of $\alpha-\beta$ takes place.
Also, invariance under interchanging $\alpha$ and $\beta$ holds because
$\pounds_1\bigl(\alpha/(\alpha-\beta)\bigr)
\equiv
\pounds_1\bigl(\beta/(\beta-\alpha)\bigr)
\pmod{p}$
and
$\pounds_1(\alpha)=\pounds_1(\beta)$,
according to Equation~\eqref{eq:L^1}.

Note that Equations~\eqref{M3} and~\eqref{CM3} relate to their infinite counterparts,
Equations~\eqref{B1} and~\eqref{B5},
much in the same way as the congruences given in Theorems~\ref{thm:central_pol} and~\ref{thm:Catalan_pol} relate to their infinite counterparts,
the generating functions of the central binomial coefficients and the Catalan numbers.

\begin{proof}
One may obtain Equation~\eqref{M3} from Equation~\eqref{M1} by means of the transform given in Equation~\eqref{eq:transform_truncated},
in the same way as in Section~\ref{sec:series} we deduced Equation~\eqref{B1} from Equation~\eqref{B4} using the transform of Equation~\eqref{eq:transform}.
 %, as in~\cite[Theorem~1]{Boy}.
However, keeping with the spirit of this paper we rather deduce the congruence from the corresponding identity of power series,
Equation~\eqref{B1}.

After rewriting Equation~\eqref{B1} in the equivalent form
\[
(1-2\beta)\sum_{k=1}^{\infty}\binom{2k}{k}H_k\bigl(\beta(1-\beta)\bigr)^k
%=
%-2\Li_1\left(\frac{\beta}{2\beta-1}\right)
=
2\Li_1(2\beta)-2\Li_1(\beta)
\]
in the power series ring $\Q[[\beta]]$, we infer the congruence
\[
(1-2\beta)\sum_{k=1}^{(p-1)/2}\binom{2k}{k}H_k\bigl(\beta(1-\beta)\bigr)^k
\equiv
2\pounds_1(2\beta)-2\pounds_1(\beta)
\pmod{\beta^p},
\]
which now takes place in the polynomial ring $\Q[\beta]$.
The right-hand side is a polynomial of degree less than $p$, while the left-hand side
is a polynomial of degree not exceeding $p$.
Its term of degree $p$ is
\[
-2\beta\binom{p-1}{(p-1)/2}H_{(p-1)/2}(-\beta^2)^{(p-1)/2}
\equiv
-2H_{(p-1)/2}\beta^p
\equiv
-2\pounds_1(-1)\beta^p
\pmod{p},
\]
because
\[
H_{(p-1)/2}=2\sum_{k=1}^{p-1}\frac{1}{2k}=\pounds_1(1)+\pounds_1(-1)\equiv\pounds_1(-1)\pmod{p}.
\]
Consequently, reduction modulo $p$ yields the polynomial congruence
\[
(1-2\beta)\sum_{k=1}^{(p-1)/2}\binom{2k}{k}H_k\bigl(\beta(1-\beta)\bigr)^k
\equiv
2\pounds_1(2\beta)-2\pounds_1(\beta)
-2\pounds_1(-1)\beta^p
\pmod{p}.
\]

To turn this congruence into the more symmetric form stated in the theorem,
note that $\pounds(-1)=\pounds_1(2)=-2\pounds_1(1/2)$,
apply Equations~\eqref{eq:4-term} and~\eqref{eq:L^1} to transform half the right-hand side into
\begin{align*}
\pounds_1(2\beta)-\pounds_1(\beta)
+2\beta^p\pounds_1(1/2)
&\equiv
-(1-2\beta)^p\pounds_1\left(\frac{1-\beta}{1-2\beta}\right)
\pmod{p}
\\&\equiv
(2\beta-1)^p\pounds_1\left(\frac{\beta}{2\beta-1}\right)
\pmod{p},
\end{align*}
and finally divide both sides of the resulting congruence by $1-2\beta$.

Like Equation~\eqref{M3}, Equation~\eqref{CM3} may be obtained in several ways,
one of which is deducing it from the generating function for the corresponding series,
Equation~\eqref{B5}.
For a change, we will obtain Equation~\eqref{CM3} from Equation~\eqref{M3} through integration, noting that
\[
\frac{\mathrm{d}}{\mathrm{d}\beta}
\sum_{k=1}^{p-1}C_kH_kx^{k+1}
=
(1-2\beta)
\sum_{k=1}^{p-1}\binom{2k}{k}H_kx^k.
\]
Because of the identities
\[
x\cdot\frac{\mathrm{d}}{\mathrm{d}\beta}\pounds_1(x)
=\pounds_{0}(x)
=\frac{x-x^p}{1-x}
\]
we have
\[
(2\beta-1)^p\pounds_0\left(\frac{\beta}{2\beta-1}\right)
=
\frac{\beta(2\beta-1)^p-\beta^p(2\beta-1)}
{(\beta-1)}
\equiv
\frac{\beta-\beta^p}
{(1-\beta)}
=
\pounds_0(\beta)
\pmod{p}
\]
and, consequently,
\[
(1-2\beta)^{p+1}\cdot\frac{\mathrm{d}}{\mathrm{d}\beta}\pounds_1\left(\frac{\beta}{2\beta-1}\right)
=
\frac{(1-2\beta)^p}{\beta}
\cdot\pounds_0\left(\frac{\beta}{2\beta-1}\right)
\equiv
-\pounds_0(\beta)
\pmod{p}.
\]
It follows that
\begin{align*}
&
\frac{\mathrm{d}}{\mathrm{d}\beta}\left(
\pounds_1(\beta)+(1-2\beta)^{p+1}\pounds_1\left(\frac{\beta}{2\beta-1}\right)
\right)
\\&\qquad\equiv
\frac{\pounds_{0}(\beta)}{\beta}
+(1-2\beta)^p\cdot\frac{\mathrm{d}}{\mathrm{d}\beta}\left(
(1-2\beta)\pounds_1\left(\frac{\beta}{2\beta-1}\right)
\right)
\pmod{p}
\\&\qquad\equiv
-2(1-2\beta)^p\pounds_1\left(\frac{\beta}{2\beta-1}\right)
\pmod{p}.
\end{align*}
This proves that the difference of the two sides of Equation~\eqref{CM3}, which is clearly a polynomial in $\Z_p[\beta]$
with no constant term and of degree not exceeding $p+1$,
has zero derivative modulo $p$.
Hence that difference can be taken to be an integral multiple of $p\beta^p$.
However, because both sides of the congruence vanish for $\beta=1$, that difference must be the zero polynomial modulo $p$.
\end{proof}

Moving on to congruences of level two, we have already produced the simplest one, which is Equation~\eqref{N2}.
We conclude this section with establishing three further congruences of level two.
The first two are analogues of Equations~\eqref{S2} and~\eqref{S3}.

\begin{theorem}
Let $p>3$ be a prime.
The following congruences hold in the polynomial ring $\Z_p[\beta]$, where $x=\beta(1-\beta)$ and $\alpha=1-\beta$:
\begin{equation}\label{C10}
\sum_{k=1}^{p-1} \binom{2k}{k}H_k^{(2)} x^{k}
\equiv
\frac{2\pounds_2(\alpha)-2\pounds_2(\beta)}{\beta-\alpha}
\pmod{p},
\end{equation}

\begin{equation}\label{C11}
\sum_{k=1}^{p-1}C_k H_k^{(2)}x^{k+1}
\equiv 2\alpha\pounds_2(\alpha)+2\beta\pounds_2(\beta)
\pmod{p}.
\end{equation}
\end{theorem}

\begin{proof}
We will deduce the congruence in Equation~\eqref{C10}
from the corresponding identity in Equation~\eqref{S2}.
Thus, we start with rewriting that in the equivalent form
\[
(1-2\beta)\sum_{k=1}^{\infty}\binom{2k}{k}H_k^{(2)} \bigl(\beta(1-\beta)\bigr)^k
=
2\Li_2(\beta)+\Li_1(\beta)^2
\]
in the power series ring $\Q[[\beta]]$.
Truncating the series and using Equation~\eqref{eq:Mirimanoff} we deduce the congruence
\[
(1-2\beta)\sum_{k=1}^{(p-1)/2}\binom{2k}{k}H_k^{(2)} \bigl(\beta(1-\beta)\bigr)^k
=
2\pounds_2(\beta)-2\pounds_2(1-\beta)
\pmod{(\beta^p,p)}
\]
in $\Z_p[\beta]$.
Now both sides are polynomials of degree less than $p$, and
hence the congruence actually holds modulo $p$ as well.
Equation~\eqref{C10} follows after dividing by $1-2\beta$ and recalling
that $\binom{2k}{k}$ vanishes modulo $p$ for $p/2<k<p$.

In a similar fashion, we will deduce Equation~\eqref{C11} from Equation~\eqref{S3}.
After rewriting the latter as
\[
\sum_{k=1}^{\infty}C_kH_k^{(2)}\bigl(\beta(1-\beta)\bigr)^{k+1}
=
2\beta\Li_2(\beta)-(1-\beta)\Li_1(\beta)^2
\]
in the formal power series ring $\Q[[\beta]]$,
because $\Li_d(\beta)/\beta\equiv\pounds_d(\beta)/\beta\pmod{\beta^{p-1}}$
%using Equation~\eqref{eq:Mirimanoff}
we deduce the congruence
\[
\sum_{k=1}^{(p-1)/2}C_kH_k^{(2)}\bigl(\beta(1-\beta)\bigr)^{k+1}
\equiv
2\beta\pounds_2(\beta)-(1-\beta)\pounds_1(\beta)^2
\pmod{\beta^{p+1}}.
\]
Because both sides are polynomials with $p$-integral rational coefficients, they can be viewed modulo $p$.
Now Equation~\eqref{eq:Mirimanoff} only yields
$\pounds_1(\beta)^2\equiv 2\pounds_2(1-\beta)\pmod{(\beta^p,p)}$.
However, this congruence actually holds modulo $(\beta^{p+1},p)$,
because of the more precise Equation~\eqref{eq:L^2} together with the fact that
$\pounds_2(1)\equiv 0\pmod{p}$.
(Alternatively, use the sharper version of Equation~\eqref{eq:Mirimanoff} given
in~\cite[Lemma~3.2]{MatTau:polylog} and the fact that the Bernoulli number $B_{p-2}$
vanishes, again for $p>3$.)
Thus, our congruence takes the form
\[
\sum_{k=1}^{(p-1)/2}C_kH_k^{(2)}\bigl(\beta(1-\beta)\bigr)^{k+1}
\equiv
2\beta\pounds_2(\beta)+2(1-\beta)\pounds_2(1-\beta)
\pmod{(\beta^{p+1},p)}
\]
in $\Z_p[\beta]$.
Because both sides are polynomials of degree not exceeding $p$, this congruence actually holds modulo $p$, rather than just modulo $(\beta^{p+1},p)$.
Finally, because $C_k\equiv 0\pmod{p}$ for $p/2<k<p-1$, and $H_{p-1}(2)\equiv 0\pmod{p}$,
we can extend the summation range to $0<k<p$, and our congruence takes the desired final form.
\end{proof}

Another congruence of level two is an analogue of Equation~\eqref{B2}, namely,
\begin{equation}\label{T}
\sum_{k=1}^{p-1}\binom{2k}{k}\frac{H_k}{k}x^k
\equiv
2(\alpha-\beta)^{p}\left(
\pounds_2\left(\frac{\alpha}{\alpha-\beta}\right)
-\pounds_2\left(\frac{\beta}{\beta-\alpha}\right)
\right)
\pmod{p},
\end{equation}
also to be read in the polynomial ring $\Z_p[\beta]$, where $x=\beta(1-\beta)$ and $\alpha=1-\beta$, for $p>2$.
Note that invariance of the right-hand side under interchanging $\alpha$ and $\beta$ is manifest in this case.
It may be possible to deduce Equation~\eqref{T} from Equation~\eqref{B2} by truncation using similar arguments to those used so far,
but it seems simpler to obtain it from Equation~\eqref{C10} by means of the sequence transform of Equation~\eqref{eq:transform_truncated}.
In fact, this procedure is followed in~\cite{Tauraso:Fibonacci}, where Equation~\eqref{T}
appears as~\cite[Equation~(13)]{Tauraso:Fibonacci}.
The same paper contains a similar derivation of our Equation~\eqref{M3}, which appears there as~\cite[Equation~(12)]{Tauraso:Fibonacci}.

\section{Some applications to numerical congruences}\label{sec:applications}

In this last section we present some examples of interesting numerical congruences which can be obtained by evaluating some of
our polynomial congruences to particular values of $x$.
We limit ourselves to a selection of the most elegant ones.

In Equation~\eqref{S4} we have found a closed form for the sum of the two series
$\sum_{k=1}^{\infty}
\binom{2k}{k}H_k^{(2)}x^k/k$
and
$\sum_{k=1}^{\infty}
\binom{2k}{k}x^k/k^3$,
but it may not be possible to do so for the individual series.
However, closed forms modulo $p$ can be found for the two analogous finite sums,
thus refining Equation~\eqref{N3}.
In fact, an evaluation modulo $p$ of the former was essentially found in~\cite{MatTau:polylog}, namely,
\begin{equation}\label{C12}
\sum_{k=1}^{p-1} \binom{2k}{k}\frac{H_k^{(2)}}{k} x^{k}
\equiv
-2x^p(\pounds_3(1-1/\alpha)+\pounds_3(1-1/\beta))
\pmod{p},
\end{equation}
for $p>3$.
To see this, according to~\cite[Equation~(38)]{MatTau:polylog} we have
\[
\sum_{k=1}^{p-1} \binom{2k}{k}\frac{H_k^{(2)}}{k} x^{k}
\equiv
-2x^p
\sum_{k=1}^{p-1} \frac{v_k(2-1/x)}{k^3}
\pmod{p},
\]
where
$
\sum_{k=1}^{p-1} v_k(y)/k^d
=\pounds_d(\gamma)+\pounds_d(\gamma^{-1})
$
with
$\gamma^2-y\gamma+1=0$,
as described in~\cite[Section~5]{MatTau:polylog} with slight notational changes.
Because the roots of the quadratic equation
$\gamma^2-(2-1/x)\gamma+1=0$
are precisely our $1-1/\alpha$ and $1-1/\beta$,
Equation~\eqref{C12} follows.
Now, subtracting Equation~\eqref{C12} from Equation~\eqref{N3} we find
\begin{equation}\label{C13}
\sum_{k=1}^{p-1} \binom{2k}{k}\frac{x^{k}}{k^3}
\equiv
4\pounds_3(\alpha)+4\pounds_3(\beta)
+2x^p(\pounds_3(1-1/\alpha)+\pounds_3(1-1/\beta))
\pmod{p}.
\end{equation}

Equation~\eqref{C13} is noteworthy because it had appeared unaccessible to the methods of~\cite{MatTau:polylog},
where we only obtained versions with lower powers of $k$ at the denominator in place of $k^3$, see~\cite[Theorem~7.1]{MatTau:polylog}.
Thus, we can now supplement the sample numerical congruences given in~\cite[Section~8]{MatTau:polylog}
with some obtained through specializations of the polynomial congruence of Equation~\eqref{C13}, such as
\begin{align}
&\label{eq:num1}
\sum_{k=1}^{p-1}\binom{2k}{k}\frac{1}{k^3}\equiv
\frac{2B_{p-3}}{3}\pmod{p},
\\&\label{eq:num2}
\sum_{k=1}^{p-1}\binom{2k}{k}\frac{(-2)^k}{k^3}\equiv
-\frac{8q_p(2)^3+31B_{p-3}}{12}\pmod{p},
\\&\label{eq:num3}
\sum_{k=1}^{p-1}\binom{2k}{k}\frac{1}{4^k k^3}\equiv
\frac{4q_p(2)^3+2B_{p-3}}{3}\pmod{p},
\\&\label{eq:num4}
\sum_{k=1}^{p-1}\frac{(-1)^k}{k^3}
\binom{2k}{k}\left(L_{3k}-1\right)\equiv
-\frac{3}{5}\left(q_L^3+2B_{p-3}\right)\pmod{p};
\end{align}
the first three congruences hold for any prime $p>3$, and the fourth one requires $p>5$.
Here $B_n$ denotes a Bernoulli number, $q_p(2)=(2^{p-1}-1)/2$ is a Fermat quotient,
$L_n$ denotes a Lucas number,
and $q_L=(L_p-1)/p$ is a Lucas quotient.
Equation~\eqref{eq:num1} appeared in~\cite[Theorem 2]{HP12},
but the others appear to be new.
For the reader's convenience, the following table lists the required values of $\alpha$, $\beta$, and the related quantities which appear in Equation~\eqref{C13}.

\setlength{\extrarowheight}{5pt}
\begin{center}
\begin{tabular}{|c|c|c|c|c|}
\hline
$x$ & $\aa$ & $\bb$ &$1-1/\aa$ &$1-1/\bb$\\
\hline
$1$ & $\omega_6$ & $\omega_6^{-1}$ & $\omega_6$ & $\omega_6^{-1}$
\\
$-2$ & $2$ & $-1$& $1/2$ & $2$
\\
$1/4$ & $1/2$ & $1/2$ & $-1$ & $-1$
\\
$-1$ & $\phi_+$ & $\phi_-$ & $\phi_-^2$ & $\phi_+^2$
\\
$-\phi^3_{+}$ & $\phi_+^2$ &$-\phi_+$ &$-\phi_-$ &$\phi_+$
\\
$-\phi^3_{-}$ & $-\phi_-$ & $\phi_-^2$ &$\phi_-$ &$-\phi_+$
\\[3pt] \hline
\end{tabular}\end{center}
Here $\omega_6$ is a primitive complex sixth root of unity, and $\phi_{\pm}=(1\pm\sqrt{5})/2$.
Equations~\eqref{eq:num1}, \eqref{eq:num2}, and~\eqref{eq:num3} follow at once by appropriately evaluating Equation~\eqref{C13}
and using the congruences produced in~\cite[Section~4]{MatTau:polylog}
for the required values of $\pounds_3$.
Equations~\eqref{eq:num4} is slightly more complicated and requires three evaluations of Equation~\eqref{C13},
corresponding to the three last rows of the above table.
In fact, because $L_{n}=\phi^{n}_{+}+\phi^{n}_{-}$ the left-hand side of Equation~\eqref{eq:num4} can be written as
\[
%\sum_{k=1}^{p-1}\frac{(-1)^k}{k^3}
%\binom{2k}{k}\left(L_{3k}-1\right)=
\sum_{k=1}^{p-1}\frac{(-\phi^3_+)^k}{k^3}\binom{2k}{k}
+\sum_{k=1}^{p-1}\frac{(-\phi^3_-)^k}{k^3}\binom{2k}{k}
-\sum_{k=1}^{p-1}\frac{(-1)^k}{k^3}\binom{2k}{k}.
\]
According to Equation~\eqref{C13} this is congruent, modulo $p$, to a certain linear combination of values of $\pounds_3$.
The desired conclusion follows from congruences for some of those values given in~\cite[Theorem~4.4]{MatTau:polylog}
and an application of standard congruences for polylogarithms recalled in~\cite[Section~2]{MatTau:polylog}.
Unfortunately, we are not able to supplement Equations~\eqref{eq:num1}--\eqref{eq:num4} with a similar evaluation of the sum
$\sum_{k=1}^{p-1}\bigl((-1)^k/k^3\bigr)\binom{2k}{k}$,
for example, because from~\cite{MatTau:polylog} we know evaluations modulo $p$ of $\pounds_3(\phi_{\pm}^2)$, but not of $\pounds_3(\phi_{\pm})$.

\bibliography{References}
\def\cprime{$'$} \def\polhk#1{\setbox0=\hbox{#1}{\ooalign{\hidewidth
  \lower1.5ex\hbox{`}\hidewidth\crcr\unhbox0}}}

\vspace{0.5cm}
\end{document}